\documentclass[12pt,reqno]{amsart}
\usepackage{etoolbox}

\makeatletter
\patchcmd\maketitle
{\uppercasenonmath\shorttitle}
{}
{}{}
\patchcmd\maketitle
{\@nx\MakeUppercase{\the\toks@}}
{\the\toks@}
{}
{}{}
\patchcmd\@settitle{\uppercasenonmath\@title}{\Large}{}{}
\patchcmd\@setauthors
{\MakeUppercase{\authors}}
{\authors}
{}{}
\makeatother
\usepackage{amsmath,amssymb,amsthm, amsfonts}
\usepackage{color}
\usepackage{url}
\usepackage{combelow}
\usepackage{enumerate}
\usepackage{palatino}
\usepackage{euler}
\usepackage{mathrsfs}
\usepackage[utf8]{inputenc}
\usepackage[T1]{fontenc}
\newtheorem{theorem}{Theorem}[section]

\newtheorem{corollary}{Corollary}[section]
\newtheorem{proposition}{Proposition}[section]
\newtheorem{lemma}{Lemma}[section]
\newtheorem{remark}{Remark}[section]
\newtheorem{example}{Example}[section]
\numberwithin{equation}{section}

\newcommand\norm[1]{\left\lVert#1\right\rVert}

\newcommand\skal[2]{\left\langle #1,#2\right\rangle}

\usepackage[colorlinks=true]{hyperref}
\hypersetup{urlcolor=blue, citecolor=red , linkcolor= blue}
\usepackage{cite}

\begin{document}
		\title[On the $\rho$-numerical radius of rank-one operators]{On the $\rho$-numerical radius of rank-one operators and parametrized Buzano-type inequalities}
		
		
		\author[H. Stankovi\'c]{Hranislav Stankovi\'c}
		\address{Faculty of Electronic Engineering, University of Ni\v s, Aleksandra Medvedeva 4, Ni\v s, Serbia
		}
		\email{\url{hranislav.stankovic@elfak.ni.ac.rs}}
		
		
		\begin{abstract}
			Let $\mathcal{H}$ be a complex Hilbert space and $\rho>0$. We establish the explicit formula
			\[
			\omega_\rho (a\otimes b)=\frac{1}{\rho}\|a\| \|b\|+\left|1-\frac{1}{\rho}\right||\langle a,b\rangle|,\qquad a,b\in\mathcal{H},
			\]
			for the $\rho$-numerical radius of rank-one operators, a unified expression interpolating between the operator norm, the numerical radius, and the spectral radius, which correspond to $\rho=1,2$, and $\rho\to\infty$, respectively. As an application, we derive a parametrized family of Buzano-type inequalities for four vectors. We extract from it an explicit closed-form bound and, under a reality condition which is automatically satisfied in real Hilbert spaces, we determine the best bound in the family in closed form. Examples show that the resulting bounds can be strictly sharper than both the Cauchy--Schwarz and the Buzano inequality, and the classical Buzano inequality is recovered as a boundary case, which yields an alternative proof of it.
		\end{abstract}
		
		\maketitle
		
		\bigskip
		
		\section{Introduction}
		
		\bigskip
		
		Let $(\mathcal{H}, \langle \cdot,\cdot\rangle)$ be a complex Hilbert space with the norm $\|\cdot\|$, and denote by $\mathfrak{B}(\mathcal{H})$ the $C^{*}$-algebra of all bounded linear operators on $\mathcal{H}$.
		We also let $\mathbb{D}$ denote the open unit disk in $\mathbb{C}$,
		while $\overline{\mathbb{D}}$ denotes the closure of $\mathbb{D}$, that is
		\[
		\overline{\mathbb{D}} = \{z \in \mathbb{C} \colon |z| \leq 1\}.
		\]
		
		For $T\in\mathfrak{B}(\mathcal{H})$, $\|T\|$ will denote the operator norm of $T$, $\sigma(T)$ will denote the spectrum of $T$, while $r(T)$ will signify the spectral radius of $T$.
		The numerical range of $T\in\mathfrak{B}(\mathcal{H})$ is defined as
		\[
		\mathcal{W}(T) = \{\langle Tx, x \rangle \colon x \in \mathcal{H},\, \|x\| = 1\},
		\]
		while the numerical radius is
		\[
		\omega(T) = \sup_{w \in \mathcal{W}(T)} |w|.
		\]
		
		For a comprehensive study of the numerical radius, we refer the reader to the excellent monographs \cite{GauWu21, GustafsonRao97}.
		\medskip
		
		Following \cite{SzNagyFoias66}, for $\rho>0$,
		an operator \( T \in \mathfrak{B}(\mathcal{H})\) is called a $\rho$-contraction (or $T$ is said to be of class \( \mathcal{C}_\rho \)) if there exists a unitary operator \( U\in\mathfrak{B}(\mathcal{K}) \), for some Hilbert space \( \mathcal{K} \) such that \( \mathcal{K}\supseteq \mathcal{H} \), and
		\[
		T^n h = \rho P_{\mathcal{H}} U^n h, \quad \text{for all } h \in \mathcal{H} \text{ and } n=1,2,\ldots,
		\]
		where $P_{\mathcal{H}}$ is the orthogonal projection from $\mathcal{K}$ onto $\mathcal{H}$.
		The $\rho$-numerical radius of an operator $T\in\mathfrak{B}(\mathcal{H})$ is defined in \cite{Holbrook68} by
		\[
		\omega_\rho(T) = \inf\left\{u>0 : \frac{1}{u}T \text{ is a $\rho$-contraction}\right\}.
		\]
		In other words, $\omega_\rho(\cdot)$ is the Minkowski functional of the class $\mathcal{C}_\rho$. For $T\in\mathfrak{B}(\mathcal{H})$, the family $\{\omega_\rho(T)\}_{\rho>0}$ interpolates between classical operator radii. More precisely,
		\[
		\omega_1(T) = \|T\|, \qquad \omega_2(T)=\omega(T), \quad \text{ and }\quad
		\omega_\infty(T):=\lim_{\rho\to\infty}\omega_\rho(T) = r(T).
		\]
		
		It is well known that $\omega_\rho(\cdot)$ is a self-adjoint and weakly unitarily invariant quasinorm on $\mathfrak{B}(\mathcal{H})$, that the function $\rho\mapsto\omega_\rho(T)$ is non-increasing, and that
		\[
		r(T)\leq\omega_\rho(T)\qquad\text{and}\qquad \frac{1}{\rho}\|T\|\leq \omega_\rho(T)\leq \max\Big\{1,\frac{2}{\rho}-1\Big\}\|T\|,\qquad T\in\mathfrak{B}(\mathcal{H}).
		\]
		Moreover, for certain classes of operators, explicit formulae for $\omega_\rho(\cdot)$ are available. For instance, if $T$ is normal, then
		\[
		\omega_\rho(T) =
		\begin{cases}
			(2\rho^{-1}-1)\|T\|, & 0<\rho<1,\\[0.3em]
			\|T\|, & \rho\geq 1,
		\end{cases}
		\]
		and if $T$ is $2$-nilpotent, then $\omega_\rho(T)=\tfrac{1}{\rho}\|T\|$.  For further details, see \cite{Holbrook68, Holbrook71, AndoNishio73}. A recent survey \cite{WuGau25} is also of great interest regarding this topic.
		
		We also mention that there are other notable generalizations of the numerical radius of an operator. See, for instance, \cite{AbuOmarKittaneh19, ElHaddadHirzallahKittaneh25, GauWu21, KittanehZamani24b, MarcusAndresen77, Rajic05, StankovicKrsticDamnjanovic25}.
		
		\medskip
		
		Finally, if $a, b \in \mathcal{H}$, then we denote by $a \otimes b$ the $\mathfrak{B}(\mathcal{H})$ operator of rank (at most) one defined by
		\[
		(a\otimes b)x = \langle x,a \rangle\, b,\quad x\in\mathcal{H}.
		\]
		For an operator of the form $a \otimes b$, we have $(a\otimes b)^* = b \otimes a$, $\|a\otimes b\| = \|a\|\|b\|$, and $r(a\otimes b)=|\langle a, b\rangle|$. Based on \cite[Theorem 1]{FujiiKubo93}, we have the identity
		\begin{equation}\label{eq:w1}
			\omega(a\otimes b) = \frac{\|a\| \, \|b\|+|\langle a,b\rangle|}{2}.
		\end{equation}
		From \eqref{eq:w1}, it is immediate that the inequality
		\begin{equation}\label{eq:buzano_ineq}
			|\langle a, x\rangle \langle x, b\rangle|\leq \frac{\|a\| \, \|b\|+|\langle a,b\rangle|}{2}\norm{x}^2
		\end{equation}
		holds for all $a, b, x \in\mathcal{H}$. Inequality \eqref{eq:buzano_ineq} is known as the Buzano inequality (see \cite{Buzano74}) and it generalizes the Cauchy--Schwarz inequality. The original proof is somewhat intricate as it requires specific facts about the orthogonal decomposition of a complete inner product space. A simpler proof can be found in \cite{FujiiKubo93}. For various refinements and generalizations of the Buzano inequality in recent years, see \cite{BottazziConde23, DencicStankovicKrstic26, Dragomir85, KittanehZamani24a, StankovicDiss}.
		
		\bigskip
		
		Motivated by these foundational results, the objectives of this paper are the following. First, we establish an explicit formula for the $\rho$-numerical radius $\omega_\rho(a\otimes b)$, valid for all $\rho > 0$, thereby providing a natural generalization of the identity \eqref{eq:w1}. Second, combining this formula with a block-matrix representation of $\omega_\rho(\cdot)$ due to Kittaneh and Zamani \cite{KittanehZamani24}, we derive a family of Buzano-type inequalities involving four vectors and a parameter $\lambda\in[-1,1]$, which reduces to \eqref{eq:buzano_ineq} for $\lambda=\pm1$ and $x=0$. Third, we extract from this family two explicit inequalities: the best possible one under a mild reality condition on the vectors involved, which is automatically satisfied in real Hilbert spaces, and a closed-form one valid without any additional assumption. Examples show that the resulting bounds can be strictly sharper than both the Cauchy--Schwarz and the Buzano inequality.
		
		\bigskip
		\section{Main results}
		\bigskip
		
		Before presenting the proof of our primary formula for the $\rho$-numerical radius of rank-one operators, we require two auxiliary results. The first is a characterization theorem of Ando and Nishio \cite{AndoNishio73}, linking the $\rho$-numerical radius to resolvent-type bounds.
		
		\begin{theorem}\cite{AndoNishio73}\label{thm:alpha_ineq}
			Let \( T \in \mathfrak{B}(\mathcal{H}) \), $\rho>0$ and $\alpha>0$. Then,
			\[
			\omega_\rho(T) \leq \alpha \iff \left\| T \left[(\rho - 1)T - \rho \alpha z I\right]^{-1} \right\| \leq  1 \quad \text{for all } z \notin \overline{\mathbb{D}}.
			\]
		\end{theorem}
		
		The second is a computational lemma giving an explicit formula for the resolvent of a rank-one operator, which will simplify the norm expression arising from Theorem \ref{thm:alpha_ineq}.
		
		\begin{lemma}\label{lem:inverse_f}
			Let $a,b\in\mathcal{H}$, and let $\lambda\notin\{0,\langle b, a\rangle\}$. Then
			\begin{equation*}
				(a\otimes b-\lambda I)^{-1}=\frac{1}{\lambda}\frac{1}{\langle b, a\rangle-\lambda}(a\otimes b)-\frac{1}{\lambda}I.
			\end{equation*}
		\end{lemma}
		
		\begin{proof}
			Since $\sigma(a\otimes b)=\{0,\langle b, a\rangle\}$ and $\lambda\notin\{0,\langle b, a\rangle\}$, the operator $a\otimes b-\lambda I$ is invertible. Let us denote the proposed inverse by
			\[
			T:=\frac{1}{\lambda}\frac{1}{\langle b,a\rangle-\lambda}\,(a\otimes b)-\frac{1}{\lambda}I.
			\]
			It is sufficient to verify that \((a\otimes b-\lambda I)T=I\). By utilizing the standard rank-one composition identity
			\begin{equation}\label{eq:square}
				(a\otimes b)^2=\langle b,a\rangle(a\otimes b),
			\end{equation}
			we obtain
			\begin{align*}
				(a\otimes b-\lambda I)T
				&= \frac{1}{\lambda}\frac{1}{\langle b,a\rangle-\lambda}(a\otimes b)^2
				-\frac{1}{\lambda}(a\otimes b) -\left(\frac{1}{\langle b,a\rangle-\lambda}\,(a\otimes b)-I\right)\\
				&= \frac{1}{\lambda}\frac{\langle b,a\rangle}{\langle b,a\rangle-\lambda}(a\otimes b)
				-\frac{1}{\lambda}(a\otimes b)-\frac{1}{\langle b,a\rangle-\lambda}\,(a\otimes b)+I \\
				&= \frac{1}{\langle b,a\rangle-\lambda}\,(a\otimes b)-\frac{1}{\langle b,a\rangle-\lambda}\,(a\otimes b)+I = I,
			\end{align*}
			which completes the proof.
		\end{proof}
		
		We can now establish the main formula.
		
		\begin{theorem}\label{thm:main}
			Let $a,b\in\mathcal{H}$ and $\rho>0$. Then
			\begin{equation}\label{eq:main_equality}
				\omega_\rho (a\otimes b)=\frac{1}{\rho}\|a\| \|b\|+\left|1-\frac{1}{\rho}\right||\langle a,b\rangle|.
			\end{equation}
		\end{theorem}
		
		\begin{proof}
			If $a=0$ or $b=0$, equality \eqref{eq:main_equality} holds trivially, so we assume throughout that $a$ and $b$ are non-zero. We first treat the case $\rho>1$, beginning with normalized vectors $\|a\|=\|b\|=1$; the general case will follow by homogeneity, and the case $0<\rho\leq1$ by means of an identity of Ando and Nishio. Let us define the quantity
			\[
			\alpha:=\frac{1}{\rho}+\left(1-\frac{1}{\rho}\right)|\langle a,b\rangle|.
			\]
			We aim to show that $\omega_\rho(a\otimes b)\leq \alpha$. By Theorem~\ref{thm:alpha_ineq}, this is equivalent to verifying that
			\[
			\left\|(a\otimes b)\Big[(\rho-1)(a\otimes b)-\rho \alpha z I\Big]^{-1}\right\|\leq 1,
			\qquad \text{for all } z\notin\overline{\mathbb{D}},
			\]
			or, written differently,
			\begin{equation}\label{eq:inverse_intermediate}
				\left\|(a\otimes b)\Big[a\otimes b-\tfrac{\rho\alpha z}{\rho-1}I\Big]^{-1}\right\|\leq \rho-1,
				\qquad \text{for all } z\notin\overline{\mathbb{D}}.
			\end{equation}
			Set $\lambda := \frac{\rho\alpha z}{\rho-1}$ and note that Lemma~\ref{lem:inverse_f} is applicable. Indeed, since $\rho>1$, we have $\alpha\geq\frac{1}{\rho}>0$, so that $\lambda\neq0$ because $|z|>1$. Moreover, if $\lambda=\langle b,a\rangle$, then $z=\frac{(\rho-1)\langle b,a\rangle}{\rho\alpha}$, and since $\rho\alpha=1+(\rho-1)|\langle a,b\rangle|$, this would give
			\[
			|z|=\frac{(\rho-1)|\langle a,b\rangle|}{1+(\rho-1)|\langle a,b\rangle|}<1,
			\]
			which contradicts $z\notin\overline{\mathbb{D}}$. By Lemma~\ref{lem:inverse_f} and \eqref{eq:square}, we have
			\begin{align*}
				(a\otimes b)\Big[a\otimes b-\lambda I\Big]^{-1}
				&=(a\otimes b)\Big[\tfrac{1}{\lambda(\langle b,a\rangle-\lambda)}\,a\otimes b-\tfrac{1}{\lambda}I\Big]\\
				&=\frac{\langle b,a\rangle}{\lambda(\langle b,a\rangle-\lambda)}\,(a\otimes b)-\frac{1}{\lambda}\,(a\otimes b)\\
				&=\frac{\langle b,a\rangle-(\langle b,a\rangle-\lambda)}{\lambda(\langle b,a\rangle-\lambda)}\,(a\otimes b)\\&
				=\frac{1}{\langle b,a\rangle-\lambda}\,(a\otimes b).
			\end{align*}
			Since $\|a\otimes b\|=\|a\|\|b\|=1$, inequality \eqref{eq:inverse_intermediate} therefore reduces to
			\[
			1\leq (\rho-1)|\langle b,a\rangle-\lambda|,\qquad z\notin\overline{\mathbb{D}}.
			\]
			Since $$\lambda=\frac{1}{\rho-1}\big(1+(\rho-1)|\langle a,b\rangle|\big)z,$$ it remains to show that
			\[
			1\leq |(\rho-1)\langle b,a\rangle-(1+(\rho-1)|\langle b,a\rangle|)z|, \qquad z\notin\overline{\mathbb{D}}.
			\]
			Let $w:=(\rho-1)\langle b,a\rangle$. By the reverse triangle inequality, we then have
			\[
			|w-(1+|w|)z|\geq (1+|w|)|z|-|w|=|z|+(|z|-1)|w|\geq |z|>1,
			\]
			which establishes the upper bound $\omega_\rho(a\otimes b)\leq \alpha$.
			
			\medskip
			
			To prove the reverse inequality, we again invoke Theorem~\ref{thm:alpha_ineq}, which guarantees that
			\[
			\left\|(a\otimes b)\Big[(\rho-1)(a\otimes b)-\rho \omega_\rho(a\otimes b)z I\Big]^{-1}\right\|\leq 1,
			\qquad z\notin\overline{\mathbb{D}}.
			\]
			We now set $\lambda:=\frac{\rho\,\omega_\rho(a\otimes b)z}{\rho-1}$ and note that Lemma~\ref{lem:inverse_f} is applicable again. Indeed, since $|z|>1$ and $\omega_\rho(a\otimes b)\geq\frac{1}{\rho}\|a\|\|b\|>0$, we have $\lambda\neq0$. Moreover, as $\frac{\rho}{\rho-1}>1$ and $|z|>1$,
			\[
			|\lambda|=\frac{\rho}{\rho-1}\,\omega_\rho(a\otimes b)\,|z|>\omega_\rho(a\otimes b)\geq r(a\otimes b)=|\langle b,a\rangle|,
			\]
			so that $\lambda\neq\langle b,a\rangle$. Following the exact same algebraic reduction as before, this condition yields
			\begin{equation}\label{eq:wr_geq_1}
				1\leq |(\rho-1)\langle b,a\rangle-\rho \omega_\rho(a\otimes b)z|,
				\qquad z\notin\overline{\mathbb{D}}.
			\end{equation}
			Let $\theta = \arg(\langle b,a\rangle)$ if $\langle b,a\rangle\neq 0$, and $\theta = 0$ otherwise. By choosing $z'=te^{i\theta}$ with $t>1$, we ensure $z'\notin\overline{\mathbb{D}}$. Substituting this specific $z'$ into \eqref{eq:wr_geq_1} gives
			\[
			1\leq \big|(\rho-1)|\langle b,a\rangle|-\rho \omega_\rho(a\otimes b)\, t\big|,\qquad t>1.
			\]
			Since $\omega_\rho(a\otimes b)\geq r(a\otimes b)=|\langle b,a\rangle|$, we have
			\[
			\rho\,\omega_\rho(a\otimes b)\, t\geq \rho\,\omega_\rho(a\otimes b)\geq\rho|\langle b,a\rangle|\geq(\rho-1)|\langle b,a\rangle|,\qquad t>1,
			\]
			so the expression under the modulus is non-positive, and the previous inequality reads
			\[
			1\leq \rho \omega_\rho(a\otimes b)\, t-(\rho-1)|\langle b,a\rangle|,\qquad t>1.
			\]
			Letting $t\to1^+$, we obtain
			\[
			1\leq \rho \omega_\rho(a\otimes b)-(\rho-1)|\langle b,a\rangle|,
			\]
			which simplifies to $$\omega_\rho(a\otimes b)\geq \frac{1}{\rho}\big(1+(\rho-1)|\langle b,a\rangle|\big)=\alpha.$$ Combining this with our upper bound, we establish that
			\begin{equation}\label{eq:formula_norm_1}
				\omega_\rho(a\otimes b)=\frac{1}{\rho}+\left(1-\frac{1}{\rho}\right)|\langle a,b\rangle|,
			\end{equation}
			whenever $\|a\|=\|b\|=1$ (and $\rho>1$).
			
			\medskip
			
			With the normalized case resolved, we can easily evaluate the case of arbitrary non-zero vectors $a,b\in\mathcal{H}$. By the positive homogeneity of the $\rho$-numerical radius and \eqref{eq:formula_norm_1}, we obtain
			\begin{align*}
				\omega_\rho(a\otimes b)
				&=\|a\|\|b\|\;\omega_\rho\!\left(\frac{a}{\|a\|}\otimes \frac{b}{\|b\|}\right)\\
				&=\|a\|\|b\|\left(\frac{1}{\rho}+\Big(1-\frac{1}{\rho}\Big)\frac{|\langle a,b\rangle|}{\|a\|\|b\|}\right)\\
				&=\frac{1}{\rho}\|a\|\|b\|+\left(1-\frac{1}{\rho}\right)|\langle a,b\rangle|,
			\end{align*}
			which extends the validity of the formula to all  vectors.
			
			\medskip
			
			Finally, we address the case where $0<\rho< 1$. As observed by Ando and Nishio \cite{AndoNishio73}, the $\rho$-numerical radius satisfies the identity $$\rho\,\omega_\rho(\cdot)=(2-\rho)\,\omega_{2-\rho}(\cdot)$$ for $0<\rho<2$. Applying this identity to our proven formula (since $2-\rho >1$), we have
			\begin{align*}
				\rho\,\omega_\rho(a\otimes b)
				&=(2-\rho)\left[\frac{1}{2-\rho}\|a\|\|b\|+\Big(1-\frac{1}{2-\rho}\Big)|\langle a,b\rangle|\right]\\
				&=\|a\|\|b\|+(1-\rho)|\langle a,b\rangle|.
			\end{align*}
			Dividing both sides by $\rho$ yields
			\[
			\omega_\rho(a\otimes b)
			=\frac{1}{\rho}\|a\|\|b\|+\Big(\frac{1}{\rho}-1\Big)|\langle a,b\rangle|
			=\frac{1}{\rho}\|a\| \|b\|+\left|1-\frac{1}{\rho}\right||\langle a,b\rangle|.
			\]
			Since the case $\rho=1$ is trivial, we have established the equality for all $\rho>0$.
		\end{proof}

		As a direct consequence of the explicit formula for the $\rho$-numerical radius of rank-one operators, we can derive a parametrized inequality involving four arbitrary vectors.
		
		\begin{lemma}\label{lem:param}
			Let $a, b,x,y\in\mathcal{H}$ and $\lambda\in[-1,1]$. Then,
			\begin{equation}\label{eq:main_cor}
				\left|\langle a, y\rangle\langle  \sqrt{1-\lambda^2}x+\lambda y, b\rangle\right|\leq \frac{\|a\| \|b\|+ |\lambda||\langle a, b\rangle|}{2}\left(\|x\|^2+\|y\|^2\right).
			\end{equation}
		\end{lemma}
		
		\begin{proof}
			Let $\lambda\in[-1,1)$ and put $\rho:=1-\lambda\in(0,2]$. Let $a,b\in\mathcal{H}$. According to \cite[Theorem 3.1]{KittanehZamani24}, we have that
			\begin{equation*}
				\omega_\rho(a\otimes b)=\frac{2}{\rho}\omega(\widetilde{a\otimes b}),
			\end{equation*}
			where
			\begin{equation*}
				\widetilde{a\otimes b}=\begin{bmatrix}
					0&\sqrt{\rho(2-\rho)}(a\otimes b)\\
					0&(1-\rho) (a\otimes b)
				\end{bmatrix}.
			\end{equation*}
			Using \eqref{eq:main_equality}, we have that
			\begin{equation}\label{eq:omega_mat}
				\omega(\widetilde{a\otimes b})=\frac{\rho}{2}\cdot\left(\frac{1}{\rho}\|a\| \|b\|+\left|1-\frac{1}{\rho}\right||\langle a,b\rangle|\right)=\frac{\|a\| \|b\|+ |1-\rho||\langle a, b\rangle|}{2}.
			\end{equation}
			
			Now let $x,y\in\mathcal{H}$ be such that $\|x\|^2+\|y\|^2=1$, and let $\mathbf{x}=\begin{bmatrix} x& y \end{bmatrix}^\top$. Then, $\norm{\mathbf{x}}=1$, and
			\begin{align*}
				\langle \widetilde{a\otimes b}\, \mathbf{x},\mathbf{x}\rangle&=\skal{\begin{bmatrix}
						0&\sqrt{\rho(2-\rho)}(a\otimes b)\\
						0&(1-\rho) (a\otimes b)	\end{bmatrix}\begin{bmatrix}
						x\\y
				\end{bmatrix}}{\begin{bmatrix}
						x\\y
				\end{bmatrix}} \\
				&=\skal{\begin{bmatrix}
						\sqrt{\rho(2-\rho)}(a\otimes b)y\\
						(1-\rho)(a\otimes b)y
				\end{bmatrix}}{\begin{bmatrix}
						x\\y
				\end{bmatrix}}\\
				&=\skal{\begin{bmatrix}
						\sqrt{\rho(2-\rho)}\langle y, a\rangle b\\
						(1-\rho) \langle y, a\rangle b
				\end{bmatrix}}{\begin{bmatrix}
						x\\y
				\end{bmatrix}}\\
				&=\sqrt{\rho(2-\rho)}\langle y, a\rangle\langle b,x\rangle+(1-\rho)\langle y, a\rangle\langle b,y\rangle\\
				&=\langle y, a\rangle\langle b, \sqrt{\rho(2-\rho)}x+(1-\rho) y\rangle.
			\end{align*}
			By the definition of the numerical radius and \eqref{eq:omega_mat}, it follows that
			\begin{align*}
				\left|\langle y, a\rangle\langle b, \sqrt{\rho(2-\rho)}x+(1-\rho) y\rangle\right|&\leq \sup_{\norm{\mathbf{x}}=1}\left|\langle \widetilde{a\otimes b}\, \mathbf{x},\mathbf{x}\rangle\right|=\omega(\widetilde{a\otimes b})\\
				&=\frac{\|a\| \|b\|+ |1-\rho||\langle a, b\rangle|}{2}.
			\end{align*}
			Recalling that $\rho=1-\lambda$, we have
			\[
			\sqrt{\rho(2-\rho)}=\sqrt{2\rho-\rho^2}=\sqrt{1-(1-\rho)^2}=\sqrt{1-\lambda^2}.
			\]
			Hence,
			\begin{equation}\label{eq:norms_square_1}
				\left|\langle a, y\rangle\langle  \sqrt{1-\lambda^2}x+\lambda y, b\rangle\right|\leq \frac{\|a\| \|b\|+ |\lambda||\langle a, b\rangle|}{2}.
			\end{equation}
			Now, let $x,y\in\mathcal{H}$ be not both zero, and put $N:=\|x\|^2+\|y\|^2$. Applying \eqref{eq:norms_square_1} to $x'=\frac{x}{\sqrt N}$ and $y'=\frac{y}{\sqrt N}$, for which $\|x'\|^2+\|y'\|^2=1$, we obtain
			
			\begin{align*}
				\frac{1}{N}\left|\langle a, y\rangle\langle  \sqrt{1-\lambda^2}x+\lambda y, b\rangle\right|
				&=\left|\langle a, y'\rangle\langle  \sqrt{1-\lambda^2}x'+\lambda y', b\rangle\right|\\&
				\leq \frac{\|a\| \|b\|+ |\lambda||\langle a, b\rangle|}{2},
			\end{align*}
			and multiplying by $N$ gives \eqref{eq:main_cor}. Obviously, if $x=y=0$, the inequality holds trivially.
			
			We have thus proved \eqref{eq:main_cor} for every $\lambda\in[-1,1)$. Since both sides of \eqref{eq:main_cor} are continuous functions of $\lambda$ on $[-1,1]$, the case $\lambda=1$ follows by letting $\lambda\to1^-$.
		\end{proof}
		
		In particular, the Buzano inequality \eqref{eq:buzano_ineq} is a boundary case of Lemma~\ref{lem:param}, which thus provides another proof of it.
		
		\begin{corollary}\label{cor:buzano}
			For any $a, b, y \in \mathcal{H}$,
			\[
			|\langle a, y\rangle \langle y, b\rangle| \leq \frac{\|a\|\|b\| + |\langle a,b\rangle|}{2} \|y\|^2.
			\]
		\end{corollary}
		
		\begin{proof}
			Take $\lambda=1$ and $x=0$ in \eqref{eq:main_cor}. (The choice $\lambda=-1$ leads to the same inequality.)
		\end{proof}
		
		Lemma~\ref{lem:param} provides, for every $\lambda\in[-1,1]$ such that $\langle \sqrt{1-\lambda^2}x+\lambda y, b\rangle\neq0$, the estimate
		\begin{equation}\label{eq:family_bound}
			|\langle a, y\rangle|\leq \frac{\|x\|^2+\|y\|^2}{2}\cdot\frac{\|a\| \|b\|+ |\lambda||\langle a, b\rangle|}{\left|\langle  \sqrt{1-\lambda^2}x+\lambda y, b\rangle\right|},
		\end{equation}
		and it is natural to ask for the best of these estimates, that is, for the minimum of the right-hand side of \eqref{eq:family_bound} over all admissible $\lambda$. Whenever the number $\langle x,b\rangle\langle b,y\rangle$ is real, which is automatically the case in real Hilbert spaces (see Remark~\ref{rem:real}), this minimum can be computed explicitly.
		
		\begin{theorem}\label{thm:optimal}
			Let $a,b,x,y\in\mathcal{H}$ be such that $\langle x,b\rangle\langle b,y\rangle\in\mathbb{R}$ and $(\langle x,b\rangle,\langle y,b\rangle)\neq(0,0)$. If $|\langle a,b\rangle||\langle x,b\rangle|>\|a\|\|b\||\langle y,b\rangle|$, put
			\[
			\mu=\frac{\|a\|\|b\|}{|\langle x,b\rangle|},
			\]
			and if $|\langle a,b\rangle||\langle x,b\rangle|\leq\|a\|\|b\||\langle y,b\rangle|$, put
			\[
			\mu=\frac{|\langle a,b\rangle||\langle y,b\rangle|+\sqrt{\big(\|a\|^2\|b\|^2-|\langle a,b\rangle|^2\big)|\langle x,b\rangle|^2+\|a\|^2\|b\|^2|\langle y,b\rangle|^2}}{|\langle x,b\rangle|^2+|\langle y,b\rangle|^2}.
			\]
			Then
			\begin{equation}\label{eq:optimal_bound}
				|\langle a,y\rangle|\leq\frac{\|x\|^2+\|y\|^2}{2}\,\mu .
			\end{equation}
			Moreover, $\mu$ is the minimum of the quotient
			\begin{equation}\label{eq:quotient}
				\frac{\|a\| \|b\|+ |\lambda||\langle a, b\rangle|}{\left|\langle  \sqrt{1-\lambda^2}x+\lambda y, b\rangle\right|}
			\end{equation}
			over all $\lambda\in[-1,1]$ for which the denominator is non-zero. In other words, \eqref{eq:optimal_bound} is the best inequality that can be obtained from \eqref{eq:family_bound}.
		\end{theorem}
		
		\begin{proof}
			Throughout the proof we write $\alpha=\|a\|\|b\|$, $\beta=|\langle a,b\rangle|$, $X=|\langle x,b\rangle|$ and $Y=|\langle y,b\rangle|$, so that
			\[
			\mu=\begin{cases}
				\dfrac{\alpha}{X}, & \text{if } \beta X>\alpha Y,\\[1.2em]
				\dfrac{\beta Y+\sqrt{(\alpha^2-\beta^2)X^2+\alpha^2Y^2}}{X^2+Y^2}, & \text{if } \beta X\leq\alpha Y.
			\end{cases}
			\]
			Since $(\langle x,b\rangle,\langle y,b\rangle)\neq(0,0)$, we have $b\neq0$. If $a=0$, then $\alpha=\beta=0$, hence $\mu=0$, and all the assertions hold trivially. We therefore assume that $a\neq0$, so that $\alpha>0$, and we recall that $\beta\leq\alpha$ by the Cauchy--Schwarz inequality.
			
			We first express the denominator in \eqref{eq:quotient} through $X$ and $Y$. Since $\overline{\langle y,b\rangle}=\langle b,y\rangle$, the number $\langle x,b\rangle\overline{\langle y,b\rangle}=\langle x,b\rangle\langle b,y\rangle$ is real by hypothesis, and its modulus equals $XY$. Hence, there exists $\varepsilon\in\{-1,1\}$ such that $\langle x,b\rangle\overline{\langle y,b\rangle}=\varepsilon XY$. For $\lambda\in[-1,1]$ we therefore have
			\begin{align*}
				\left|\langle \sqrt{1-\lambda^2}x+\lambda y, b\rangle\right|^2
				&=\left|\sqrt{1-\lambda^2}\,\langle x,b\rangle+\lambda\langle y,b\rangle\right|^2\\
				&=(1-\lambda^2)|\langle x,b\rangle|^2+\lambda^2|\langle y,b\rangle|^2\\
				&\quad+2\lambda\sqrt{1-\lambda^2}\,\operatorname{Re}\big(\langle x,b\rangle\overline{\langle y,b\rangle}\big)\\
				&=(1-\lambda^2)X^2+\lambda^2Y^2+2\varepsilon\lambda\sqrt{1-\lambda^2}\,XY\\
				&=\left(\sqrt{1-\lambda^2}\,X+\varepsilon\lambda Y\right)^2,
			\end{align*}
			and consequently
			\begin{equation}\label{eq:denominator}
				\left|\langle \sqrt{1-\lambda^2}x+\lambda y, b\rangle\right|=\left|\sqrt{1-\lambda^2}\,X+\varepsilon\lambda Y\right|,\qquad\lambda\in[-1,1].
			\end{equation}
			In particular, by the triangle inequality,
			\begin{equation}\label{eq:denominator_bound}
				\left|\langle \sqrt{1-\lambda^2}x+\lambda y, b\rangle\right|\leq\sqrt{1-\lambda^2}\,X+|\lambda|Y,\qquad\lambda\in[-1,1].
			\end{equation}
			
			Next we show that
			\begin{equation}\label{eq:key_ineq}
				\alpha+\beta s\geq\mu\left(\sqrt{1-s^2}\,X+sY\right),\qquad s\in[0,1],
			\end{equation}
			and that equality holds in \eqref{eq:key_ineq} for some $s_0\in[0,1]$ with $\sqrt{1-s_0^2}\,X+s_0Y>0$.
			
			Assume first that $\beta X>\alpha Y$. Then $X>0$ and $\mu=\alpha/X$, so that $\mu Y-\beta=(\alpha Y-\beta X)/X<0$. Hence, for $s\in[0,1]$,
			\[
			\mu\left(\sqrt{1-s^2}\,X+sY\right)=\alpha\sqrt{1-s^2}+\mu Ys\leq\alpha\sqrt{1-s^2}+\beta s\leq\alpha+\beta s,
			\]
			which is \eqref{eq:key_ineq}. For $s_0=0$ both inequalities above become equalities, and $\sqrt{1-s_0^2}\,X+s_0Y=X$ is positive.
			
			Assume now that $\beta X\leq\alpha Y$, and put
			\[
			q=\sqrt{(\alpha^2-\beta^2)X^2+\alpha^2Y^2},
			\]
			which is well defined since $\beta\leq\alpha$. Then
			\[
			\mu=\frac{\beta Y+q}{X^2+Y^2}.
			\]
			We claim that $\mu Y\geq\beta$. Multiplying by $X^2+Y^2>0$ and using $\mu(X^2+Y^2)=\beta Y+q$, we see that this claim is equivalent to
			\[
			\beta Y^2+qY\geq\beta X^2+\beta Y^2,
			\]
			that is, to
			\[
			qY\geq\beta X^2.
			\]
			Both sides here are non-negative, so squaring gives the equivalent inequality
			\[
			Y^2\left((\alpha^2-\beta^2)X^2+\alpha^2Y^2\right)\geq\beta^2X^4,
			\]
			that is,
			\[
			\alpha^2Y^2(X^2+Y^2)\geq\beta^2X^2(X^2+Y^2),
			\]
			which holds because $\alpha Y\geq\beta X$. Furthermore, a direct computation gives
			\begin{align*}
				\mu^2X^2+(\mu Y-\beta)^2
				&=\mu^2(X^2+Y^2)-2\beta Y\mu+\beta^2\\
				&=\frac{(\beta Y+q)^2-2\beta Y(\beta Y+q)+\beta^2(X^2+Y^2)}{X^2+Y^2}\\
				&=\frac{q^2-\beta^2Y^2+\beta^2(X^2+Y^2)}{X^2+Y^2}\\
				&=\frac{(\alpha^2-\beta^2)X^2+\alpha^2Y^2+\beta^2X^2}{X^2+Y^2}=\alpha^2.
			\end{align*}
			Thus $(\mu X,\mu Y-\beta)$ is a vector in $\mathbb{R}^2$ with non-negative coordinates and with Euclidean norm $\alpha$. Applying the Cauchy--Schwarz inequality in $\mathbb{R}^2$ to this vector and to the unit vector $(\sqrt{1-s^2},s)$, we obtain
			\[
			\mu\sqrt{1-s^2}\,X+(\mu Y-\beta)s\leq\alpha,\qquad s\in[0,1],
			\]
			which is \eqref{eq:key_ineq}. Equality holds when $(\sqrt{1-s^2},s)=(\mu X,\mu Y-\beta)/\alpha$, that is, for $s_0=(\mu Y-\beta)/\alpha$, which belongs to $[0,1]$ because $0\leq\mu Y-\beta\leq\alpha$. For this value we have $\sqrt{1-s_0^2}=\mu X/\alpha$ and
			\[
			\sqrt{1-s_0^2}\,X+s_0Y=\frac{\mu X^2+(\mu Y-\beta)Y}{\alpha}=\frac{\mu(X^2+Y^2)-\beta Y}{\alpha}=\frac{q}{\alpha}>0 .
			\]
			Indeed, if $q=0$, then $\alpha^2Y^2=0$ and $(\alpha^2-\beta^2)X^2=0$. Since $\alpha>0$, the first relation gives $Y=0$, and then $X>0$ because $(\langle x,b\rangle,\langle y,b\rangle)\neq(0,0)$. The second relation now gives $\beta=\alpha>0$, so the assumption $\beta X\leq\alpha Y=0$ yields $X=0$, a contradiction.
			
			In both cases we have thus established \eqref{eq:key_ineq}, together with a point $s_0\in[0,1]$ at which equality holds in \eqref{eq:key_ineq} and $\sqrt{1-s_0^2}\,X+s_0Y>0$.
			
			We can now conclude the proof. Put $\lambda_0=\varepsilon s_0\in[-1,1]$. Then $|\lambda_0|=s_0$ and $\varepsilon\lambda_0=\varepsilon^2s_0=s_0$, so \eqref{eq:denominator} with $\lambda=\lambda_0$ gives
			\[
			\left|\langle \sqrt{1-\lambda_0^2}x+\lambda_0 y, b\rangle\right|=\sqrt{1-s_0^2}\,X+s_0Y>0 .
			\]
			In particular, $\lambda_0$ is admissible in \eqref{eq:quotient}. Moreover, the equality case of \eqref{eq:key_ineq} reads
			\begin{equation}\label{eq:equality_s0}
				\alpha+\beta s_0=\mu\left(\sqrt{1-s_0^2}\,X+s_0Y\right).
			\end{equation}
			
			We first prove \eqref{eq:optimal_bound}. Applying \eqref{eq:main_cor} with $\lambda=\lambda_0$ and using \eqref{eq:equality_s0}, we get
			\begin{align*}
				|\langle a,y\rangle|\left(\sqrt{1-s_0^2}\,X+s_0Y\right)
				&\leq\frac{\alpha+\beta s_0}{2}\left(\|x\|^2+\|y\|^2\right)\\
				&=\frac{\|x\|^2+\|y\|^2}{2}\,\mu\left(\sqrt{1-s_0^2}\,X+s_0Y\right),
			\end{align*}
			and dividing by the positive number $\sqrt{1-s_0^2}\,X+s_0Y$ yields \eqref{eq:optimal_bound}.
			
			Next, the quotient \eqref{eq:quotient} at $\lambda=\lambda_0$ equals
			\[
			\frac{\alpha+|\lambda_0|\beta}{\left|\langle \sqrt{1-\lambda_0^2}x+\lambda_0 y, b\rangle\right|}=\frac{\alpha+\beta s_0}{\sqrt{1-s_0^2}\,X+s_0Y}=\mu ,
			\]
			again by \eqref{eq:equality_s0}. Thus the value $\mu$ is attained.
			
			Finally, let $\lambda\in[-1,1]$ be such that $\langle \sqrt{1-\lambda^2}x+\lambda y, b\rangle\neq0$, and put $s=|\lambda|$. By \eqref{eq:denominator_bound},
			\[
			0<\left|\langle \sqrt{1-\lambda^2}x+\lambda y, b\rangle\right|\leq\sqrt{1-s^2}\,X+sY,
			\]
			while \eqref{eq:key_ineq} gives $\alpha+\beta s\geq\mu\left(\sqrt{1-s^2}\,X+sY\right)$. Consequently,
			\[
			\frac{\alpha+|\lambda|\beta}{\left|\langle  \sqrt{1-\lambda^2}x+\lambda y, b\rangle\right|}\geq\frac{\alpha+\beta s}{\sqrt{1-s^2}\,X+sY}\geq\mu .
			\]
			Together with the previous step, this shows that $\mu$ is the minimum of the quotient \eqref{eq:quotient}, and the proof is complete.
		\end{proof}
		
		\begin{remark}\label{rem:real}
			Theorem~\ref{thm:optimal} applies, in particular, to arbitrary vectors of a real Hilbert space. Indeed, if $\mathcal{H}$ is a real Hilbert space, we may regard it as a subset of its complexification $\mathcal{H}\oplus i\mathcal{H}$, whose inner product restricted to $\mathcal{H}$ coincides with the given one. Applying \eqref{eq:main_cor} and \eqref{eq:optimal_bound} in the complexification to vectors $a,b,x,y\in\mathcal{H}$, we obtain the same inequalities in $\mathcal{H}$, and the number $\langle x,b\rangle\langle b,y\rangle$ is then automatically real. The same argument applies to Corollary~\ref{cor:projection_bound} below. Finally, in a complex Hilbert space the hypothesis $\langle x,b\rangle\langle b,y\rangle\in\mathbb{R}$ of Theorem~\ref{thm:optimal} is satisfied, for instance, whenever $\langle x,b\rangle$ and $\langle y,b\rangle$ are both real, or when one of them vanishes.
		\end{remark}
		
		When $\langle x,b\rangle\langle b,y\rangle$ is not real, the minimization in Theorem~\ref{thm:optimal} does not appear to admit such a simple closed form. We therefore turn to a natural explicit choice of the parameter in \eqref{eq:family_bound}. Geometrically, Lemma~\ref{lem:param} provides a family of upper bounds for the mixed product $\left|\langle a, y\rangle\langle  z, b\rangle\right|$, where the vector $z = \sqrt{1-\lambda^2}\,x+\lambda y$ traces an elliptical arc in the subspace spanned by $x$ and $y$, and we choose $\lambda$ so as to maximize the second factor $|\langle z, b\rangle|$ along this arc. Since the right-hand side of \eqref{eq:main_cor} also depends on $\lambda$, this choice need not be optimal in the sense of Theorem~\ref{thm:optimal} (see Example~\ref{ex:active}), but it leads to a closed-form bound without any reality assumption. The maximization rests on the following elementary fact.
		
		\begin{lemma}\label{lem:harmonic_max}
			Let $A, B \in \mathbb{R}$ and consider the function $g(s) = A\cos s + B \sin s$ for $s \in \mathbb{R}$. Then
			\[
			\max_{s\in\mathbb{R}} g(s) = \sqrt{A^2 + B^2},
			\]
			and, when $(A,B) \neq (0,0)$, the maximum is attained precisely at those $s_0$ with $\cos s_0 = A/\sqrt{A^2+B^2}$ and $\sin s_0 = B/\sqrt{A^2+B^2}$.
		\end{lemma}
		\begin{proof}
			If $(A,B)=(0,0)$ the statement is trivial. Otherwise write $(A,B) = \sqrt{A^2+B^2}\,(\cos\varphi, \sin\varphi)$ for a unique $\varphi \in (-\pi,\pi]$. Then
			\[
			g(s) = \sqrt{A^2+B^2}\,\bigl(\cos\varphi \cos s + \sin\varphi \sin s\bigr) = \sqrt{A^2+B^2}\,\cos(s - \varphi),
			\]
			by the cosine addition formula. Since $\cos(\cdot)$ has maximum value $1$, attained when its argument is an integer multiple of $2\pi$, we get $\max_s g(s) = \sqrt{A^2+B^2}$, attained precisely when $s \equiv \varphi \pmod{2\pi}$. For such $s_0$ one has $\cos s_0 = \cos\varphi = A/\sqrt{A^2+B^2}$ and $\sin s_0 = \sin\varphi = B/\sqrt{A^2+B^2}$.
		\end{proof}
		
		\begin{corollary}\label{cor:projection_bound}
			Let $a, b, x, y \in \mathcal{H}$ and assume that the quantity
			\[
			\Delta = \sqrt{\left(|\langle x,b\rangle|^2+|\langle y,b\rangle|^2\right)^2 - 4\left(\operatorname{Im}(\langle x,b\rangle\langle b,y\rangle)\right)^2}
			\]
			is strictly positive. Then
			\begin{equation}\label{eq:projection_bound}
				|\langle a, y\rangle| \leq \frac{\left(\|x\|^2+\|y\|^2\right) \Big( \|a\|\|b\| + \delta |\langle a,b\rangle| \Big)}{\sqrt{2\left(|\langle x,b\rangle|^2+|\langle y,b\rangle|^2+\Delta\right)}},
			\end{equation}
			where the parameter $\delta$ is defined by
			\[
			\delta = \sqrt{\frac{\Delta - |\langle x,b\rangle|^2 + |\langle y,b\rangle|^2}{2\Delta}}.
			\]
		\end{corollary}
		
		\begin{proof}
			By squaring both sides of inequality \eqref{eq:main_cor}, we obtain
			\begin{equation}\label{eq:squared_cor}
				|\langle a, y\rangle|^2\, |\langle \sqrt{1-\lambda^2}x+\lambda y, b\rangle|^2 \leq \frac{(\|x\|^2+\|y\|^2)^2}{4} \Big(\|a\| \|b\|+ |\lambda|\,|\langle a, b\rangle|\Big)^2 .
			\end{equation}
			Parameterize $\lambda \in [-1, 1]$ by $\lambda = \sin t$, so that $\sqrt{1-\lambda^2} = \cos t$ for $t \in [-\tfrac{\pi}{2}, \tfrac{\pi}{2}]$. Writing $X = \langle x,b\rangle$ and $Y = \langle y,b\rangle$, the second factor on the left-hand side becomes
			\[
			f(t) := |X \cos t + Y \sin t|^2 = |X|^2 \cos^2 t + |Y|^2 \sin^2 t + 2\operatorname{Re}(X \overline{Y}) \sin t \cos t .
			\]
			Using the double-angle identities $2\cos^2 t = 1+\cos 2t$, $2\sin^2 t = 1-\cos 2t$ and $2\sin t \cos t = \sin 2t$, this reads
			\[
			f(t) = \frac{|X|^2+|Y|^2}{2} + \frac{|X|^2-|Y|^2}{2} \cos 2t + \operatorname{Re}(X \overline{Y}) \sin 2t .
			\]
			By Lemma~\ref{lem:harmonic_max}, applied with $s=2t$, $A = \tfrac{|X|^2-|Y|^2}{2}$ and $B = \operatorname{Re}(X\overline{Y})$, the maximum over $t$ of the oscillatory part is $\sqrt{A^2+B^2}$, where
			\[
			A^2 + B^2 = \frac{(|X|^2-|Y|^2)^2 + 4\bigl(\operatorname{Re}(X \overline{Y})\bigr)^2}{4} = \frac{(|X|^2+|Y|^2)^2 - 4\bigl(\operatorname{Im}(X \overline{Y})\bigr)^2}{4} = \frac{\Delta^2}{4} ,
			\]
			the middle equality following from $|X|^2|Y|^2 = (\operatorname{Re} X\overline{Y})^2 + (\operatorname{Im} X\overline{Y})^2$. Hence
			\[
			\max_{t} f(t) = \frac{|X|^2+|Y|^2}{2} + \frac{\Delta}{2} = \frac{|\langle x,b\rangle|^2+|\langle y,b\rangle|^2+\Delta}{2} .
			\]
			Since $\Delta > 0$ by assumption, the amplitude $\sqrt{A^2+B^2} = \Delta/2$ is nonzero, so by Lemma~\ref{lem:harmonic_max} there exists $s_0\in(-\pi,\pi]$ with $\cos s_0 = A/(\Delta/2) = (|X|^2-|Y|^2)/\Delta$ at which the maximum of $A\cos s+B\sin s$ is attained. Setting $t_0=s_0/2\in\left(-\tfrac{\pi}{2}, \tfrac{\pi}{2}\right]$, the maximum of $f$ is attained at $t_0$, and $\cos2t_0=(|X|^2-|Y|^2)/\Delta$. The corresponding parameter is $\lambda_0 = \sin t_0$, and
			\[
			\lambda_0^2 = \sin^2 t_0 = \frac{1 - \cos 2t_0}{2} = \frac{\Delta - |\langle x,b\rangle|^2 + |\langle y,b\rangle|^2}{2\Delta} = \delta^2 ,
			\]
						so $|\lambda_0| = \delta$. Substituting $\lambda=\lambda_0$ into \eqref{eq:squared_cor} and using $f(t_0)=\max_{t}f(t)$ together with $|\lambda_0|=\delta$, we obtain
			\[
			|\langle a, y\rangle|^2 \cdot \frac{|\langle x,b\rangle|^2+|\langle y,b\rangle|^2+\Delta}{2} \leq \frac{(\|x\|^2+\|y\|^2)^2}{4} \Big(\|a\|\|b\| + \delta\,|\langle a,b\rangle|\Big)^2 .
			\]
			Taking square roots and dividing by the positive number
			\[
			\sqrt{\frac{|\langle x,b\rangle|^2+|\langle y,b\rangle|^2+\Delta}{2}}
			\]
			yields \eqref{eq:projection_bound}.
		\end{proof}
		
				\begin{remark}
						The condition $\Delta = 0$ is excluded from Corollary \ref{cor:projection_bound} as it corresponds to a degenerate configuration. Indeed, assume that $\Delta=0$, that is,
			\[
			\left(|\langle x,b\rangle|^2+|\langle y,b\rangle|^2\right)^2=4\left(\operatorname{Im}(\langle x,b\rangle\langle b,y\rangle)\right)^2.
			\]
			Since $\left|\operatorname{Im}(\langle x,b\rangle\langle b,y\rangle)\right|\leq\left|\langle x,b\rangle\langle b,y\rangle\right|=|\langle x,b\rangle||\langle y,b\rangle|$, it follows that
			\[
			\left(|\langle x,b\rangle|^2+|\langle y,b\rangle|^2\right)^2\leq4|\langle x,b\rangle|^2|\langle y,b\rangle|^2,
			\]
			that is, $\left(|\langle x,b\rangle|^2-|\langle y,b\rangle|^2\right)^2\leq0$. Hence $|\langle x,b\rangle|=|\langle y,b\rangle|=:m$, and the assumption $\Delta=0$ becomes
			\[
			4m^4=4\left(\operatorname{Im}(\langle x,b\rangle\langle b,y\rangle)\right)^2 ,
			\]
			i.e.,
			\[
			m^4=\left(\operatorname{Im}(\langle x,b\rangle\langle b,y\rangle)\right)^2 .
			\]
			On the other hand,
			\[
			\left(\operatorname{Re}(\langle x,b\rangle\langle b,y\rangle)\right)^2+\left(\operatorname{Im}(\langle x,b\rangle\langle b,y\rangle)\right)^2=\left|\langle x,b\rangle\langle b,y\rangle\right|^2=m^4 ,
			\]
			which gives $\operatorname{Re}(\langle x,b\rangle\langle b,y\rangle)=0$. 			If $\langle y,b\rangle=0$, then also $\langle x,b\rangle=0$. If $\langle y,b\rangle\neq0$, write $\langle x,b\rangle=c\,\langle y,b\rangle$, where $c=\langle x,b\rangle/\langle y,b\rangle\in\mathbb{C}$. Then
			\[
			|c|=\frac{|\langle x,b\rangle|}{|\langle y,b\rangle|}=1
			\qquad\text{and}\qquad
			\langle x,b\rangle\langle b,y\rangle=c\,\langle y,b\rangle\overline{\langle y,b\rangle}=c\,|\langle y,b\rangle|^2,
			\]
			so that
			\[
			0=\operatorname{Re}(\langle x,b\rangle\langle b,y\rangle)=\operatorname{Re}(c)\,|\langle y,b\rangle|^2,
			\]
			and hence $\operatorname{Re}(c)=0$. Thus, $c=\pm i$, and therefore $\langle x, b \rangle = \pm i \langle y, b \rangle$ in both cases.
			
			 In this case the amplitude of the oscillatory part of $f$ vanishes, so $f$ is constant, that is, $|\langle \sqrt{1-\lambda^2}x + \lambda y, b\rangle|=|\langle x,b\rangle|$ for every $\lambda\in[-1,1]$. The left-hand side of \eqref{eq:main_cor} is therefore the same for all $\lambda$, while the right-hand side is non-decreasing in $|\lambda|$. Consequently, the inequality corresponding to $\lambda=0$ implies all the others, and it reads
			\[
			|\langle a, y \rangle|\, |\langle x, b \rangle| \leq \frac{\|x\|^2+\|y\|^2}{2} \|a\|\|b\| ,
			\]
			which is also an immediate consequence of the Cauchy--Schwarz and AM--GM inequalities. The content of Corollary~\ref{cor:projection_bound} therefore lies in the regime $\Delta > 0$.
		\end{remark}
		
				For non-zero vectors $u,v\in\mathcal{H}$, the cosine of the angle between $u$ and $v$ is defined as
		\[
		\cos\angle(u,v)=\frac{|\langle u,v\rangle|}{\|u\|\|v\|}\in[0,1].
		\]
		The following result, stated in terms of these cosines, describes precisely when the family \eqref{eq:family_bound} can improve on the Cauchy--Schwarz inequality, and how to choose $x$ in that case.
		
		\begin{proposition}\label{prop:regime}
			Let $a,b,y\in\mathcal{H}$ be non-zero vectors.
			\begin{enumerate}[\rm($i$)]
				\item If $\cos\angle(a,b)\geq\cos\angle(y,b)$, then the right-hand side of \eqref{eq:family_bound} is at least $\|a\|\|y\|$ for every $x\in\mathcal{H}$ and every $\lambda\in[-1,1]$ such that $\langle\sqrt{1-\lambda^2}x+\lambda y,b\rangle\neq0$.
				\item If $\cos\angle(a,b)<\cos\angle(y,b)$, let $x=\frac{\|y\|}{\|b\|}\,\theta b$, where $\theta\in\mathbb{C}$, $|\theta|=1$, is chosen so that $\langle x,b\rangle\langle b,y\rangle\geq0$. Then the number $\mu$ from Theorem~\ref{thm:optimal} satisfies $$\|y\|\mu<\|a\|,$$ so that \eqref{eq:optimal_bound} is strictly sharper than the Cauchy--Schwarz inequality.
			\end{enumerate}
		\end{proposition}
		
		\begin{proof}			Replacing $a$ by $ra$ and $b$ by $\sigma b$ with $r,\sigma>0$ multiplies both sides of \eqref{eq:family_bound} and of \eqref{eq:optimal_bound}, as well as the quantity $\|a\|\|y\|$, by $r$, and leaves the cosines $\cos\angle(a,b)$ and $\cos\angle(y,b)$ unchanged. Hence, we may assume that $\|a\|=\|b\|=1$. Put $\beta=|\langle a,b\rangle|$ and $c=|\langle y,b\rangle|/\|y\|$, so that $0\leq\beta\leq1$ and $0\leq c\leq1$. The assumption in ($i$) reads $\beta\geq c$, while the one in ($ii$) reads $\beta<c$.
			
			We first prove ($i$). Let $x\in\mathcal{H}$ and $\lambda\in[-1,1]$ be such that $\langle\sqrt{1-\lambda^2}x+\lambda y,b\rangle\neq0$, and put $s=|\lambda|$ and $t=\|x\|/\|y\|$. By the triangle and the Cauchy--Schwarz inequalities,
			\[
			\left|\langle\sqrt{1-\lambda^2}x+\lambda y,b\rangle\right|\leq\sqrt{1-s^2}\,\|x\|+s|\langle y,b\rangle|=\|y\|\left(\sqrt{1-s^2}\,t+sc\right),
			\]
						and consequently, since $\|x\|^2+\|y\|^2=\|y\|^2(t^2+1)$, the right-hand side of \eqref{eq:family_bound} is at least
			\[
			\frac{\|x\|^2+\|y\|^2}{2}\cdot\frac{1+s\beta}{\|y\|\left(\sqrt{1-s^2}\,t+sc\right)}=\|y\|\cdot\frac{(t^2+1)(1+s\beta)}{2\left(\sqrt{1-s^2}\,t+sc\right)}.
			\]
						Since $\|a\|\|y\|=\|y\|$, it suffices to show that
			\[
			(t^2+1)(1+s\beta)\geq2\left(\sqrt{1-s^2}\,t+sc\right).
			\]
			As $c\leq\beta$ and $s\geq0$, we have $2sc\leq2s\beta$, so it is enough to prove that
			\[
			(t^2+1)(1+s\beta)\geq2\sqrt{1-s^2}\,t+2s\beta ,
			\]
			that is, after expanding the left-hand side and rearranging,
			\[
			(1+s\beta)\,t^2-2\sqrt{1-s^2}\,t+(1-s\beta)\geq0 .
			\]
			The left-hand side is a quadratic polynomial in $t$ with positive leading coefficient, whose discriminant equals
			\[
			4(1-s^2)-4(1+s\beta)(1-s\beta)=-4s^2(1-\beta^2)\leq0 .
			\]
			Hence it is non-negative for every real $t$, which proves ($i$).
			
						We now prove ($ii$). Since $c>\beta\geq0$, we have $c>0$, hence $\langle y,b\rangle\neq0$. Recalling that $\|b\|=1$, we have $x=\|y\|\theta b$ and
			\[
			\langle x,b\rangle\langle b,y\rangle=\|y\|\,\theta\,\langle b,y\rangle ,
			\]
			so the choice $\theta=\overline{\langle b,y\rangle}/|\langle b,y\rangle|$ satisfies $|\theta|=1$ and gives $$\langle x,b\rangle\langle b,y\rangle=\|y\|\,|\langle b,y\rangle|>0.$$ In particular, $\langle x,b\rangle\langle b,y\rangle$ is real and $\langle x,b\rangle=\|y\|\theta\neq0$, so the hypotheses of Theorem~\ref{thm:optimal} are satisfied. In the notation of the proof of Theorem~\ref{thm:optimal} we have $\alpha=1$, $X=|\langle x,b\rangle|=\|y\|$ and $Y=|\langle y,b\rangle|=c\|y\|$. Thus $$\beta X=\beta\|y\|<c\|y\|=\alpha Y,$$ so the second case of Theorem~\ref{thm:optimal} applies and
			\[
			\mu=\frac{\beta c\|y\|+\sqrt{(1-\beta^2)\|y\|^2+c^2\|y\|^2}}{\|y\|^2(1+c^2)}=\frac{\beta c+\sqrt{1-\beta^2+c^2}}{\|y\|(1+c^2)} .
			\]
			Since $\|x\|=\|y\|$, the right-hand side of \eqref{eq:optimal_bound} equals $\|y\|^2\mu$, and the inequality $\|y\|^2\mu<\|a\|\|y\|=\|y\|$ is equivalent to
			\[
			\sqrt{1-\beta^2+c^2}<1+c^2-\beta c .
			\]
			The right-hand side here is positive, because $\beta c\leq c\leq1+c^2$. Squaring, we obtain the equivalent inequality $1-\beta^2+c^2<(1+c^2-\beta c)^2$, and a direct expansion gives
			\[
			(1+c^2-\beta c)^2-(1-\beta^2+c^2)=c^4-2\beta c^3+(1+\beta^2)c^2-2\beta c+\beta^2=(c^2-\beta c)^2+(c-\beta)^2 ,
			\]
			which is positive because $c\neq\beta$. This proves ($ii$).
		\end{proof}
		
				\begin{remark}\label{rem:orthogonal}
						In the case $a\perp b$ the choice of $x$ along $b$ can be optimized completely. Assume that $\|b\|=1$ and $y\not\perp b$, put $Y=|\langle y,b\rangle|>0$, and consider $x=t\theta b$ with $t\geq0$ and $\theta$ as in Proposition~\ref{prop:regime}. Since $\langle a,b\rangle=0$, the second case of Theorem~\ref{thm:optimal} applies, and $\|x\|=|\langle x,b\rangle|=t$, so that
			\[
			\mu=\frac{\sqrt{\|a\|^2t^2+\|a\|^2Y^2}}{t^2+Y^2}=\frac{\|a\|}{\sqrt{t^2+Y^2}} .
			\]
			Hence the right-hand side of \eqref{eq:optimal_bound} equals
			\[
			\frac{t^2+\|y\|^2}{2}\cdot\frac{\|a\|}{\sqrt{t^2+Y^2}}=\frac{\|a\|}{2}\left(\sqrt{t^2+Y^2}+\frac{\|y\|^2-Y^2}{\sqrt{t^2+Y^2}}\right)\geq\|a\|\sqrt{\|y\|^2-Y^2}
			\]
			by the AM--GM inequality, with equality precisely when $t=\sqrt{\|y\|^2-2Y^2}$. Such a $t\geq0$ exists exactly when $Y\leq\|y\|/\sqrt2$, and in that case \eqref{eq:optimal_bound} becomes
			\[
			|\langle a,y\rangle|\leq\|a\|\sqrt{\|y\|^2-Y^2}.
			\]
								This estimate cannot be improved: for any $\gamma>0$ and any $\eta,Y$ with $0<Y\leq\eta$ there exist vectors $a,y$ with $\|a\|=\gamma$, $\|y\|=\eta$, $|\langle y,b\rangle|=Y$ and $a\perp b$ for which equality holds. Indeed, take a unit vector $w\perp b$ and put $y=Yb+\sqrt{\eta^2-Y^2}\,w$ and $a=\gamma w$; then $|\langle a,y\rangle|=\gamma\sqrt{\eta^2-Y^2}$.
		\end{remark}
		
		The following two examples illustrate Proposition~\ref{prop:regime} and Remark~\ref{rem:orthogonal}. In the first one, $a\perp b$ and $x$ is the vector from Remark~\ref{rem:orthogonal}, so that the resulting bound is exact. In the second one, $\langle a,b\rangle\neq0$, the vector $x$ is the one from Proposition~\ref{prop:regime}($ii$), and the resulting bounds are strictly sharper than both the Cauchy--Schwarz and the Buzano inequality. All quantities are computed explicitly, so that the bounds can be compared with the exact value of $|\langle a,y\rangle|$.

		\begin{example}\label{ex:orthogonal}
			Let $\mathcal{H} = \mathbb{C}^2$ with the standard inner product, and consider the real vectors
			\[
			a = \begin{bmatrix} 1 \\ 0 \end{bmatrix}, \quad b = \begin{bmatrix} 0 \\ 1 \end{bmatrix}, \quad x = \begin{bmatrix} 0 \\ 1 \end{bmatrix}, \quad y = \begin{bmatrix} \sqrt{2} \\ 1 \end{bmatrix}.
			\]
			The relevant quantities are
			\[
			\|a\| = \|b\| = 1, \quad \|x\|^2 = 1, \quad \|y\|^2 = 3,
			\]
			\[
			\langle a, b \rangle = 0, \quad \langle a, y \rangle = \sqrt{2}, \quad \langle x, b \rangle = \langle y, b \rangle = 1 .
			\]
			Since $\langle a,b\rangle=0$, the second case of Theorem~\ref{thm:optimal} applies, and
			\[
			\mu=\frac{0+\sqrt{(1-0)\cdot 1+1\cdot1\cdot 1}}{1+1}=\frac{\sqrt2}{2}.
			\]
			Hence \eqref{eq:optimal_bound} gives $$|\langle a,y\rangle|\leq \frac{1+3}{2}\cdot\frac{\sqrt2}{2}=\sqrt2,$$ which is exact, since $|\langle a,y\rangle|=\sqrt2$. Corollary~\ref{cor:projection_bound} gives the same bound in this case. Indeed,
			\[
			\Delta = \sqrt{(1 + 1)^2 - 0} = 2 > 0, \qquad \delta = \sqrt{\frac{2 - 1 + 1}{2\cdot 2}} = \frac{1}{\sqrt{2}} ,
			\]
			and \eqref{eq:projection_bound} gives
			\[
			|\langle a,y\rangle|\leq\frac{(1+3)\left(1\cdot1+\frac{1}{\sqrt2}\cdot0\right)}{\sqrt{2(1+1+2)}}=\frac{4}{\sqrt8}=\sqrt2,
			\]
			so equality holds in \eqref{eq:projection_bound} as well.
			
			It is instructive to compare with the classical estimates. Cauchy--Schwarz gives only $$|\langle a,y\rangle| \leq \|a\|\,\|y\| = \sqrt{3} \approx 1.732.$$ The classical Buzano inequality \eqref{eq:buzano_ineq} applied to $a, y, b$ gives $$|\langle a,y\rangle\langle y,b\rangle| \leq \tfrac{1}{2}(\|a\|\|b\| + |\langle a,b\rangle|)\|y\|^2 = \tfrac{3}{2}.$$ Since $\langle y,b\rangle = 1$, this restricts $|\langle a,y\rangle| \leq 1.5$. Both are strictly weaker than the sharp value $\sqrt{2}$ furnished here, so the parametrized inequality strictly improves on both Cauchy--Schwarz and Buzano in this configuration.
		\end{example}
		
		\begin{example}\label{ex:active}
			The improvement is not confined to the case $\langle a,b\rangle=0$. Take, again in $\mathbb{C}^2$, the real vectors
			\[
			a = \begin{bmatrix} -2 \\ -2 \end{bmatrix}, \quad b = \begin{bmatrix} 1 \\ -2 \end{bmatrix}, \quad x = \begin{bmatrix} -1 \\ 2 \end{bmatrix}, \quad y = \begin{bmatrix} 1 \\ 2 \end{bmatrix} .
			\]
			Then
			\[
			\|a\|^2 = 8, \quad \|b\|^2 = 5, \quad \|x\|^2 = \|y\|^2 = 5, \quad \langle a,b\rangle = 2 \neq 0,
			\]
			\[
			\langle a,y\rangle = -6, \quad \langle x,b\rangle = -5, \quad \langle y,b\rangle = -3 .
			\]
			Since $|\langle a,b\rangle||\langle x,b\rangle|=10\leq\|a\|\|b\||\langle y,b\rangle|=3\sqrt{40}$, the second case of Theorem~\ref{thm:optimal} applies, and
			\[
			\mu=\frac{2\cdot 3+\sqrt{(40-4)\cdot 25+40\cdot 9}}{25+9}=\frac{6+\sqrt{1260}}{34}=\frac{3+3\sqrt{35}}{17}.
			\]
			Hence \eqref{eq:optimal_bound} gives
			\[
			|\langle a,y\rangle|\leq \frac{5+5}{2}\cdot\frac{3+3\sqrt{35}}{17}=\frac{15\left(1+\sqrt{35}\right)}{17}\approx 6.102,
			\]
			the minimum in Theorem~\ref{thm:optimal} being attained for $\lambda_0=(3\mu-2)/\sqrt{40}\approx 0.263$ (see the proof of Theorem~\ref{thm:optimal}). On the other hand, since all the inner products involved are real, $\Delta = |\langle x,b\rangle|^2 + |\langle y,b\rangle|^2 = 25 + 9 = 34 > 0$, and
			\[
			\delta = \sqrt{\frac{\Delta - |\langle x,b\rangle|^2 + |\langle y,b\rangle|^2}{2\Delta}} = \sqrt{\frac{34 - 25 + 9}{68}} = \sqrt{\frac{18}{68}} = \frac{3}{\sqrt{34}} ,
			\]
			so that \eqref{eq:projection_bound} gives
			\[
			|\langle a,y\rangle| \leq \frac{(5+5)\left(\sqrt{40} + \frac{3}{\sqrt{34}}\cdot 2\right)}{\sqrt{2\left(25+9+34\right)}}
			= 5\cdot\frac{\sqrt{40} + \frac{6}{\sqrt{34}}}{\sqrt{34}} \approx 6.306.
			\] The exact value is $|\langle a,y\rangle| = 6$, so both bounds are close and, more importantly, sharper than both classical estimates: Cauchy--Schwarz gives $$|\langle a,y\rangle| \leq \|a\|\|y\| = \sqrt{8}\sqrt{5} = \sqrt{40} \approx 6.32,$$ while Buzano applied to $a,y,b$ gives
			\[
			|\langle a,y\rangle| \leq \frac{(\|a\|\|b\| + |\langle a,b\rangle|)\|y\|^2}{2|\langle y,b\rangle|} = \frac{(\sqrt{40}+2)\cdot 5}{2\cdot 3} \approx 6.937 .
			\]
			Thus $6.102 < 6.306 < 6.325 < 6.937$: both bounds outperform Cauchy--Schwarz and Buzano even when $\langle a,b\rangle \neq 0$.
		\end{example}

				\bigskip
		\section*{Declarations}
		\noindent{\bf{Funding}}\\
		This work has been supported by the Ministry of Science, Technological Development and Innovation of the Republic of Serbia [Grant Number: 451-03-34/2026-03/200102].
		
		\vspace{0.5cm}
		
		\noindent{\bf{Availability of data and materials}}\\
		\noindent No data were used to support this study.
		\vspace{0.5cm}\\
		\noindent{\bf{Competing interests}}\\
		\noindent The author declares that he has no competing interests.
		\vspace{0.5cm}\\
		\noindent{\bf{Declaration of generative AI and AI-assisted technologies in the manuscript preparation process}}\\
		\noindent During the preparation of this work, the author used Claude (Anthropic) to verify the mathematical arguments, to assist in the derivation and drafting of parts of the results, and to improve the exposition; ChatGPT (OpenAI) and Gemini (Google) were additionally consulted for pre-submission checks. The author reviewed, verified and edited all generated content and takes full responsibility for the content of the published article.
		
		
		\vspace{0.5cm} 
		
	\end{document}